\documentclass[leqno,12pt]{amsart} 
\usepackage{amssymb}
\newcommand{\Real}{\mathrm{\mathbb{R}}}
\DeclareMathOperator{\Susp}{Susp}
\newcommand{\bb}{\mathbb}
\newcommand{\sphere}{\mathrm{\mathbb{S}}}
\usepackage{graphicx}
\usepackage{color}
\usepackage{hyperref}
\usepackage{multirow}
\usepackage{makecell}
\usepackage{hhline}
\usepackage{soul}
\usepackage{graphicx}
\usepackage{caption} 
\usepackage{tabu}

\theoremstyle{plain} 
\newtheorem{theorem}{\indent\sc Theorem}[section]
\newtheorem{sl_theorem}{\indent\sc Slice Theorem}[section]
\newtheorem{lemma}[theorem]{\indent\sc Lemma}
\newtheorem{corollary}[theorem]{\indent\sc Corollary}
\newtheorem{proposition}[theorem]{\indent\sc Proposition}

\theoremstyle{definition} 
\newtheorem{definition}[theorem]{\indent\sc Definition}
\newtheorem{remark}[theorem]{\indent\sc Remark}
\newtheorem{example}[theorem]{\indent\sc Example}

\begin{document}

\title[Three-dimensional Alexandrov spaces]{Closed three-dimensional Alexandrov spaces with isometric circle actions} 

\author[J. N\'u\~nez-Zimbr\'on]{Jes\'us N\'u\~nez-Zimbr\'on $^*$} 


%
\thanks{ 
$^{*}$ The author was partially supported by PAEP-UNAM and DGAPA-UNAM under project PAPIIT IN113713.
}
\address{
Departamento de Matem\'aticas \endgraf
Facultad de Ciencias, UNAM \endgraf
DF 04510 \endgraf
Mexico
}
\email{nunez-zimbron@ciencias.unam.mx}


\begin{abstract}
We obtain a topological and weakly equivariant classification of closed three-dimensional Alexandrov spaces with an effective, isometric circle action. This generalizes the topological and equivariant classifications of Raymond \cite{R} and Orlik and Raymond \cite{OR} of closed three-dimensional manifolds admitting an effective circle action. As an application, we prove a version of the Borel conjecture for closed three-dimensional Alexandrov spaces with circle symmetry.
\end{abstract}

\maketitle

\section{Introduction and results}

Alexandrov spaces (of curvature bounded below) appear naturally as generalizations of Riemannian manifolds of sectional curvature bounded below. Many results for Riemannian manifolds admit suitable generalizations to the Alexandrov setting and this class of metric spaces has been studied from several angles, including recently the use of transformation groups \cite{GG,GS,HS}.

Considering spaces with non-trivial isometry groups has been a fruitful avenue of research in Riemannian geometry \cite{G,Ko,S}. In Alexandrov geometry, this point of view  has provided information on the structure of Alexandrov spaces. In \cite{Be}, Berestovski\v{\i} showed that finite dimensional homogeneous metric spaces with a lower curvature bound are Riemannian manifolds. Galaz-Garcia and Searle studied in \cite{GS} Alexandrov spaces of cohomogeneity one (i.e.~those with an effective isometric action of a compact Lie group whose orbit space is one-dimensional) and classified them in dimensions at most $4$. In this paper, we classify the effective, isometric circle actions on closed, connected Alexandrov $3$-spaces, thus completing the classification of closed Alexandrov spaces, of dimension at most three, admitting an isometric action of a compact, connected Lie Group. A motivation for considering this problem is the general philosophy that, as in the Riemannian setting, in order to understand Alexandrov spaces, it is natural to study first those with a high degree of symmetry. In the particular instance of Alexandrov $3$-spaces, it seems that a complete understanding of the full class is yet out of reach (cf. Remark 4.3 in \cite{GG1}). However, in the presence of an isometric circle action, Alexandrov $3$-spaces gain several properties that make them more tractable.

In the topological category, Raymond obtained an equivariant classification of the effective actions of the circle on any closed, connected topological $3$-manifold \cite{R}. The orbit space of such an action is a topological $2$-manifold, possibly with boundary. Raymond proved that there is a complete set of invariants that determines each equivariant homeomorphism type:

\begin{theorem}[Raymond \cite{R}]
\label{TH:RAYMOND}
The set of all inequivalent (up to weakly equivariant homeomorphism) effective, isometric circle actions on a closed, connected topological $3$-manifold is in one-to-one correspondence with the set of unordered tuples 
\[
 (b ; (\varepsilon, g, f, t), \{ (\alpha_1, \beta_1), \ldots, (\alpha_n, \beta_n)  \} ).  
 \]
\end{theorem} 
Here, $b$ is the obstruction for the principal stratum of the action to be a trivial principal $S^1$-bundle. The symbol $\varepsilon$ takes two possible values, corresponding to the orientability of the orbit space. The genus of the orbit space is denoted by $g$. The number of connected components of the fixed point set is denoted by $f$, while $t$ is the number of $\mathbb{Z}_2$-isotropy  connected components. The  pairs $\{(\alpha_i,\beta_i)\}^{n}_{i=1}$ are the Seifert invariants associated to the exceptional orbits of the action, if any. 

Raymond also proved that the invariants in Theorem \ref{TH:RAYMOND} determine the manifold's prime decomposition when $f>0$. The topological classification without the restriction that $f>0$ was obtained by Orlik and Raymond in \cite{OR} (see also \cite{O}). 

The classification presented herein is an extension of the work of Orlik and Raymond to the class of closed, connected Alexandrov $3$-spaces. As opposed to a closed $3$-manifold, a closed Alexandrov $3$-space $X$ may have topologically singular points, i.e. points whose space of directions is homeomorphic to the real projective plane $\Real P^2$ (see Section \ref{S:Orbit_types}). By Perelman's work \cite{Pe}, the set of such points is discrete, and by compactness, finite. To account for these points, we add an additional set of invariants to those of Raymond: an unordered $s$-tuple $(r_1, r_2, \ldots, r_s)$ of even positive integers. The integer $s$ corresponds to the number of boundary components in the orbit space that contain orbits of topologically singular points. The integers $r_i$ correspond to the number of topologically singular points in the $i$th boundary component of the orbit space with orbits of topological singularities. If there are no topologically singular points we consider this $s$-tuple to be empty. With these definitions in hand, we may now state our main result. We let $\Susp(\Real P^2)$ denote the suspension of $\Real P^2$. 

\begin{theorem}
\label{teoprinc}
Let $S^1$ act effectively and isometrically on a closed, connected Alexandrov $3$-space $X$. Assume that $X$ has $2r$  topologically singular points, $r\geq0$. Then the following hold:
\begin{enumerate} 
	\item The set of inequivalent (up to weakly equivariant homeomorphism) effective, isometric circle actions on $X$ is in one-to-one correspondence with the set of unordered tuples
\[
	\left( b; (\varepsilon, g,f,t);\{ (\alpha_i, \beta_i) \}_{i=1}^{n} ; (r_1,r_2, \ldots, r_s) \right)
\]
where the permissible values for $b$, $\varepsilon$, $g$, $f$, $t$ and $\{ (\alpha_i,\beta_i) \}_{i=1}^n$, are the same as in Theorem \ref{TH:RAYMOND} and $(r_1,r_2,\ldots, r_s)$ is an unordered $s$-tuple of even positive integers $r_i$ such that $r_1 + \ldots+r_s =2r$.
	\item $X$ is weakly equivariantly homeomorphic to 
\[ 
M\# \underbrace{\Susp(\bb{R}P^2)\#\cdots \#\Susp(\bb{R}P^2) }_{r}
\]	
 where $M$ is the closed $3$-manifold given by the set of invariants
\[
  \left( b; (\varepsilon, g,f+s,t);\{ (\alpha_i, \beta_i) \}_{i=1}^{n} \right) 
\]
in Theorem~\ref{TH:RAYMOND}.
\end{enumerate}
\end{theorem}

The circle actions on the spaces $M\#\Susp(\bb{R}P^2)\#\ldots \#\Susp(\bb{R}P^2)$ in Theorem 1.2 are given in terms of an equivariant connected sum, which we construct in Section 3.1. It will follow from the construction that the actions are isometric with respect to some invariant Alexandrov metric. We point out that in the case that $X$ has no topologically singular points, i.e. $r=0$, $X$ is a topological manifold and Theorem 1.2 follows from the work of Raymond \cite{R}.

We also count the number of inequivalent effective, isometric circle actions on a closed, connected Alexandrov $3$-space $X$, by using that of the manifold $M$ appearing in (2) of Theorem \ref{teoprinc} (see Remark \ref{R:COUNT_ACTIONS}). 

We point out that, by Theorem \ref{teoprinc}, a closed, non-manifold Alexandrov $3$-space $X$ with an effective and isometric circle action has the form $M\#\Susp(\bb{R}P^2)\#\cdots\#\Susp(\bb{R}P^2)$. Therefore, the Seifert-Van Kampen Theorem and the resolution of the Poincar\'e conjecture imply that the only simply-connected, closed Alexandrov $3$-spaces with an effective, isometric circle action are the $3$-sphere $\mathbb{S}^3$ and connected sums of finitely many copies of $\Susp(\bb{R}P^2)$. On the other hand, Galaz-Garcia and Guijarro showed in \cite{GG1}, without any symmetry assumptions, that there are examples of simply-connected, closed Alexandrov $3$-spaces with topologically singular points which are not homeomorphic to connected sums of copies of $\Susp(\bb{R}P^2)$. One such example is the space $X$ obtained by taking the quotient of a flat $3$-torus $T^3$ by the involution $\iota:T^3\to T^3$ given by complex conjugation on each $\bb{S}^1$ factor. If $X$ were homeomorphic to a connected sum of copies of $\bb{S}^3$ and $\Susp(\bb{R}P^2)$, then it would admit a Riemannian orbifold metric having positive scalar curvature. This implies that the pullback metric (with respect to the canonical projection $T^3\to T^3/\iota$) on $T^3$ also has positive scalar curvature which contradicts results of Schoen and Yau (cf. Remark 4.3 in \cite{GG1}). 

Finally, we remark that closed Alexandrov $3$-spaces with an effective, isometric circle action fall within the class of collapsed Alexandrov $3$-spaces, considered by Mitsuishi and Yamaguchi in \cite{MY}. In our case, the collapse occurs along the orbits of the action and we obtain a more refined topological classification than the one in Section 5 of \cite{MY} by harnessing the presence of the circle action. 

Our paper is organized as follows. In Section~\ref{S:EQUIV_ALEX_GEOM} we recall some basic results on the geometry of isometric actions of compact Lie groups on Alexandrov spaces. In Section~\ref{S:Orbit_types} we give the topological structure of the orbit space of a closed, connected Alexandrov $3$-space with an effective, isometric circle action. We assign weights to the orbit space with isotropy information. Section~\ref{S:TOP_EQU_CLASS_DISK} contains the topological and equivariant classifications of effective, isometric $S^1$ actions on non-manifold closed, connected Alexandrov $3$-spaces, in the special case where there are no exceptional orbits and the orbit space is homeomorphic to a $2$-disk. In Section~\ref{S:CLASS_GENERAL}, we prove Theorem \ref{teoprinc}, obtaining the topological and equivariant classifications without any restrictions. The proofs of the main results in Sections \ref{S:Orbit_types}, \ref{S:TOP_EQU_CLASS_DISK} and \ref{S:CLASS_GENERAL} follow the same strategies as the corresponding results in \cite{R} and \cite{OR}. In Section~\ref{S:Borel}, as an application of Theorem~\ref{teoprinc}, we give a proof of the Borel conjecture for closed, connected Alexandrov $3$-spaces with circle symmetry.

\textbf{Acknowledgements}. This work is part of the author's Ph.D.~thesis. He would like to thank his advisors Fernando Galaz-Garcia and Oscar Palmas. The author would also like to thank Luis Guijarro for bringing the Borel conjecture to his attention, Bernardo Villarreal Herrera for useful conversations, and the Differential Geometry group at the University of M\"unster for its hospitality while part of this work was carried out. The author also thanks the referee for valuable comments on the manuscript.  

\section{Equivariant Alexandrov geometry}
\label{S:EQUIV_ALEX_GEOM}

Let $X$ be a finite-dimensional Alexandrov space. We assume that the reader is familiar with the basic theory of compact transformation groups as well as that of Alexandrov spaces. Basic references for these subjects are \cite{Br} and \cite{BBI,BGP} respectively.

 Fukaya and Yamaguchi showed in \cite{FY} that, as in the Riemannian case, the group of isometries of $X$ is a Lie group. If $X$ is compact then its isometry group is also compact (see ~\cite{DW}). The isometry group of an Alexandrov space has been further investigated in \cite{GG}. We consider isometric actions $G\times X \rightarrow X$  of a compact Lie group $G$ on $X$. We will denote the orbit of a point $x\in X$ by $G(x) \cong G/G_x$, where $G_x= \{ g\in G \ : \ gx= x	  \}$ is the isotropy subgroup of $x$ in $G$. The closed subgroup of $G$ given by $\cap_{x\in X} G_x$ is called the \textit{ineffective kernel} of the action. If the ineffective kernel is trivial, we will say that the action is \textit{effective}. From now on, we will suppose that all the actions we consider are effective. Given a subset $A\subset X$ we denote its image under the canonical projection $\pi: X\rightarrow X/G$ by $A^*$. In particular, $X^*=X/G$. It was proved in \cite{BGP} that the orbit space $X^{*}$ is an Alexandrov space with the same lower curvature bound as $X$. 
 
 We will denote the space of directions of $X$ at a point $x$ by $\Sigma_xX$. Given $A\subset \Sigma_xX$, we define the \textit{set of normal directions to $A$} as 
 \[
A^{\bot} = \{ v\in \Sigma_xX \ : \ d(v,w)= \mathrm{diam}(\Sigma_xX)/2\ \mathrm{for} \ \mathrm{all} \ w\in A   \}.  
 \] 
Let $S_x$ denote the tangent unit space to the orbit $G/G_x$. Galaz-Garcia and Searle proved in \cite {GS} that, if $\dim(G/G_x)>0$ then the set $S_x^\perp$ is a compact, totally geodesic Alexandrov subspace of $\Sigma_xX$ with curvature bounded below by $1$. Moreover, they showed that $\Sigma_xX$ is isometric to the join $S_x * S_x^\perp$ with the standard join metric and that either $S_x^\perp$ is connected or it contains exactly two points at distance $\pi$.

We now recall the Slice Theorem for isometric actions on Alexandrov spaces (see Harvey and Searle \cite{HS}). For a subset $A\subset X$, the metric ball of radius $\varepsilon$ centered on $A$ is denoted by $B_\varepsilon(A)$. The cone of an Alexandrov space $Y$ of $\mathrm{Curv}\geq 1$ is denoted by $K(Y)$ and it is assumed to have the standard cone metric. 

\begin{sl_theorem}[Harvey and Searle \cite{HS}]
\label{TH:SLICE_THEOREM}
Let a compact Lie group $G$ act isometrically on an Alexandrov space $X$. Then for all $x\in X$, there is some $\varepsilon_0>0$ such that for all $\varepsilon<\varepsilon_0$ there is an equivariant homeomorphism 
\[
G\times_{G_x}K(S_x^{\bot}) \rightarrow B_{\varepsilon}(G(x)). 
\]
\end{sl_theorem}

As a consequence of the Slice Theorem, a slice at $x$ is equivariantly homeomorphic to $K(S^{\bot}_x)$. It follows that  $\Sigma_{x*}X^*$, the space of directions at $x^*$ in $X^*$, is isometric to $S^{\bot}_x/G_x$. 
Alexandrov versions of Kleiner's isotropy Lemma and the principal orbit Theorem were proved by Galaz-Garcia and Guijarro in \cite{GG}. 

Let $G$ act isometrically on two Alexandrov spaces $X$ and $Y$. We will say that a mapping $\varphi: X\rightarrow Y$ is \textit{weakly equivariant} if for every $x\in X$ and $g\in G$ there exists an automorphism $f$ of $G$ such that $\varphi(gx)= f(g)\varphi(x)$. We will say that two actions on $X$ are \textit{equivalent} if there exists a weakly equivariant homeomorphism $\varphi:X\rightarrow X$. 

Let $(\Susp(\Real P^2),d_0)$ denote the spherical suspension of the unit round $\Real P^2$.  We will now give an example of an effective, isometric circle action on $\Susp(\Real P^2)$. It will play a central role in our examination of $S^1$-actions on closed, connected Alexandrov $3$-spaces. We will show in Section \ref{S:TOP_EQU_CLASS_DISK} that this is the only circle action that can occur on $(\Susp(\Real P^2),d_0)$ up to equivalence.

\begin{example}
\label{EX:STANDARD_ACTION}
We will say that the suspension of the standard cohomogeneity one circle action on the unit round $\Real P^2$ is the \emph{standard circle action} on $\Susp(\Real P^2)$. We will describe this action explicitly. Let $D^2$ be the unit disk in the plane with polar coordinates $(r,\theta)$. We identify the points of the form $(1,\theta)$ with $(1, \theta + \pi)$. Then each point in $\Real P^2$ is an equivalence class $[r,\theta]$ where $(r,\theta)\in D^2$. Therefore, the points of $\Susp(\Real P^2)$ are equivalence classes $[[r,\theta], t]$ with $[r,\theta]\in \Real P^2$ and $0\leq t\leq 1$. Now, for every $0\leq \varphi\leq 2\pi$ the standard action is given by $\varphi \cdot [[r,\theta], t] := [[r,\theta + \varphi], t]$. 
\end{example}

\section{Orbit types and orbit space}
\label{S:Orbit_types}

Let $X$ be a closed, connected  Alexandrov $3$-space with an effective, isometric $S^1$-action. We will only consider the case were the action has a non-empty set of fixed points. It will follow from our discussion below that closed Alexandrov $3$-spaces with effective, isometric circle actions without fixed points are topological manifolds, and this case has been completely classified by Orlik and Raymond \cite{OR,R}. In this section we will determine the topological structure of the orbit space $X^{*}$. We will also assign weights to its points with isotropy information.  

We call a point $x$ in $X$ \textit{topologically regular} if $\Sigma_xX$ is homeomorphic to $\sphere^2$ and \textit{topologically singular} if $\Sigma_xX$  is homeomorphic to $\Real P^2$. Let $SF$ be the set of topologically singular points of $X$. Observe that, since the action is isometric, singular points are mapped to singular points by the elements of $S^1$. By a theorem of Perelman (Theorem~0.2~in~\cite{Pe}), the codimension of the set of topologically singular points is at least 3. The compactness of $X$ then implies that $SF$ is a finite set. 

We have different orbit types according to the possible isotropy groups of the action. These groups are the trivial subgroup $\{e\}$, the cyclic subgroups $\mathbb{Z}_k$, $k\geq 2$ and $S^1$ itself. Therefore, orbits in $X$ are either $0$-dimensional or $1$-dimensional.  
This observation and the finiteness of $SF$ imply that topologically singular points are fixed by the action. We let $F$ be the set of fixed points of the action and $RF=F\setminus SF$. The points whose isotropy is not $S^1$ are topologically regular, therefore we can talk about a local orientation. We will say that an orbit with isotropy $\mathbb{Z}_k$ acting without reversing the local orientation is \textit{exceptional}; we will denote the set of points on exceptional orbits by $E$. An orbit with isotropy $\mathbb{Z}_2$ that acts reversing the local orientation will be called \textit{special exceptional} and the set of points on such orbits will be denoted by $SE$. The orbits with trivial isotropy will be called \textit{principal}. We collect the definitions of the orbit types in  Table \ref{TBL:ORBIT_TYPES}.

\begin{table}[hb]
\centering
\resizebox{\textwidth}{!}{%
\begin{tabu}{|c|c|c|c|}
\hline
 Orbit Type &  Notation& Isotropy & Space of Directions \\ \hline
Principal & $P$ & $\{e\}$ & \multirow{4}{*}{$\mathbb{S}^2$} \\ \tabucline[1pt]{1-3}
Exceptional & $E$ & \makecell{$\mathbb{Z}_k$ \\ (action preserves local orientation)} &  \\ \tabucline[1pt]{1-3}
Special exceptional & $SE$ & \makecell{$\mathbb{Z}_2$ \\ (action reverses local orientation)} &  \\ \tabucline[1pt]{1-3}
Topologically regular fixed points & $RF$ & $S^1$ &  \\ \tabucline[1pt]{-}
 Topologically singular fixed points & $SF$ & $S^1$ & $\mathbb{R}P^2$ \\ \tabucline[1pt]{-}
Fixed Points ($RF\cup SF$) & $F$ & $S^1$ & $\mathbb{S}^2$ or $\mathbb{R}P^2$ \\ \hline
\end{tabu}
}
\caption{Orbit types of an effective, isometric $S^1$ action on $X$.}
\label{TBL:ORBIT_TYPES}
\end{table}
We now investigate the topological structure of $X^{*}$. A small neighborhood of $x^{*} \in X^{*}$ is homeomorphic to $B_{\varepsilon}(x)^{*}$. By the conical neighborhood theorem of Perelman (Theorem~0.1~in~\cite{Pe}, see also Theorem 6.8 in \cite{K}), $B_{\varepsilon}(x)^{*}$ is homeomorphic to $K(\Sigma_{x^{*}}X^{*})$. Then, Theorem ~\ref{TH:SLICE_THEOREM} implies that $B_{\varepsilon}(x)^{*}$ is homeomorphic to $K(S_x^{\bot}/G_x)$. For a point $x^{*}\in SF^{*}$ this means that $B_{\varepsilon}(x)^{*}$ is homeomorphic to $K(\Real P^2/S^1)$. An action by homeomorphisms on $\Real P^2$ is equivalent to a linear action \cite{M,N}, therefore $\Real P^2/S^1$ is a closed interval with principal isotropy in the interior, $\mathbb{Z}_2$-isotropy at one endpoint and $S^1$-isotropy at the other endpoint. It follows that $x^{*}$ is the common endpoint of two arcs contained in the boundary of $X^{*}$. One of these arcs is contained in $SE^{*}$ and the other is contained in $F^{*}$. The topological structure of $X^*$ near topologically regular points is given in Lemma 1 of \cite{R}, which we now recall.

\begin{lemma}[Raymond \cite{R}]
The orbit space of $M^*$ of an effective action of a circle on a $3$-manifold $M$ is a $2$-manifold with boundary $F^*\cup SE^*$. Furthermore all orbits near $E^*$, $F^*$ or $SE^*$ are principal orbits.
\end{lemma}

 The orbit space $X^{*}$ is  weighted with isotropy information,  which we detail now. Let $C^*$ be a boundary component of $X^{*}$. We have the following three possibilities: $C^{*} \subseteq RF^{*}$, $C^{*}\subseteq SE^{*}$, or $C^{*}\cap SF^{*} \neq \emptyset$. The last possibility implies that $C^{*}\subseteq F^{*}\cup SE^{*}$ and that $C^{*}$ intersects $F^{*}$ and $SE^{*}$ non-trivially. The interior of $X^{*}$ is composed of principal orbits and $E^*$. A generic orbit space is shown in Figure ~\ref{FIG:GENERIC_ORBIT_SPACE}.  In this Figure, the interior of the manifold consists of principal orbits, except for the two highlighted points of exceptional isotropies $\mathbb{Z}_k$ and $\mathbb{Z}_l$. Along the boundary circles, points on solid lines have $S^1$ isotropy, while the points on dotted lines have $\mathbb{Z}_2$ isotropy. We summarize the previous discussion in the following proposition. 

\begin{proposition}
\label{PROP:ORBIT_SPACE}
Let $S^1$ act effectively and isometrically with $F\neq \emptyset$ on a closed, connected Alexandrov $3$-space $X$. Then the following hold:   

\begin{itemize}
\item[(1)] The orbit space $X^{*}$ is a $2$-manifold with boundary.

\item[(2)] The interior of $X^{*}$ consists of principal orbits except for a finite number of exceptional orbits.

\item[(3)] For each boundary component $C^*$ of $X^*$, one of the following possibilities holds: $C^*\subset RF^*$, $C^*\subset SE^*$ or $C^*\cap SF^*\neq \emptyset$.

\item[(4)] If $C^*\cap SF^*\neq \emptyset$, then $C^*\setminus SF^*$ is a finite union of $r\geq 2$ open intervals $\{I_k\}_{k=2}^r$, with each $I_k$ contained either in $RF^*$ or $SE^*$. 

\item[(5)] If $I_k\subset RF^*$, then $I_{k+1}\subset SE^*$ and if $I_k\subset SE^*$, then $I_{k+1}\subset RF^*$. 
\end{itemize}

\end{proposition}


\begin{figure}
\centering
\def\svgwidth{0.5\columnwidth} 
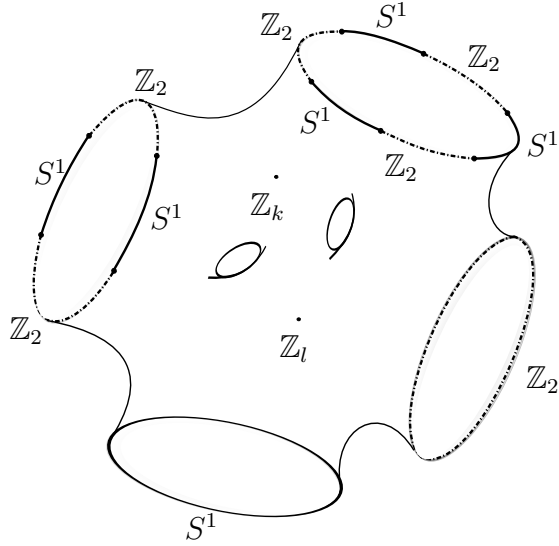
\caption{Example of an orbit space of an isometric circle action on a closed Alexandrov $3$-space. }
\label{FIG:GENERIC_ORBIT_SPACE}
\end{figure} 

We also have the following lemma. 

\begin{lemma}
\label{L:EVEN}
Let $S^1$ act effectively and isometrically on a closed, connected Alexandrov $3$-space $X$. Then $X$ has an even number of topologically singular points.
\end{lemma}

\begin{proof} If $X$ is a topological manifold, then it has $0$ topologically singular points and the result follows trivially. Therefore, we assume that the set of topologically singular points of $X$ is non-empty.

Let $C^*$ be a boundary component of $X^*$, identified with the interval $[0,2\pi]$. Let $P_r = \{0=t_1 <t_2<\ldots <t_r=2\pi \}$ be a partition of $C^*$ such that $[t_i,t_{i+1}]\subseteq F^*$ or $[t_i,t_{i+1}]\subseteq SE^*$ for each $i=1,\ldots, r$. Let $P_{\tilde{r}}$ be a minimal partition satisfying the conditions. Then $\tilde{r}>1$ if and only if $C^*\cap  SF^*\neq \emptyset$. In this case it is clear that $t_i\in SF^*$. We claim that $\tilde{r}$ is an even integer. Suppose $P_{\tilde{r}}$ has an odd number of points. Observe that adjacent intervals in $P_{\tilde{r}}$ cannot be contained both in $F^*$ or in $SE^*$ since that would make their common point superfluous, contradicting the minimality condition on $P_{\tilde{r}}$. 
\end{proof}

We remark that the conclusion of the previous Lemma holds even without the assumption of symmetry, as is observed in \cite{MY}.

We will group the topological and equivariant information of $X^*$ into a set of invariants which we list now. Let $b$ be the obstruction for the principal part of the action to be a trivial principal $S^1$-bundle. The symbol $\varepsilon$, with possible values $o$ or $n$, will stand for orientable and non-orientable $X^*$ respectively. The genus of $X^{*}$ will be denoted by an integer $g\geq0$. We let $f\geq0$ designate the number of boundary components of $X^{*}$ that are contained in $RF^{*}$. Similarly, $t\geq 0$ will stand for the number of boundary components of $X^{*}$ contained in $SE^{*}$. We associate Seifert invariants $(\alpha_i, \beta_i)$ to each exceptional orbit as in \cite{R}, (see also \cite{O}). Let $C^*_1, \ldots, C^{*}_s$ be the boundary components of $X^{*}$ that intersect $SF^*$. We define $r_i$ to be the cardinality of $C^*_i\cap SF^{*}$ for each $i=1,\ldots, s$. Note that $r_i$ is an even integer by Lemma \ref{L:EVEN}. In summary, we associate the following set of invariants to $X^{*}$:
\[ 
\left( b; (\varepsilon, g, f,t); \{ (\alpha_i, \beta_i) \}_{i=1}^{n} ; (r_1,r_2,\ldots, r_s)  \right).  
\]
In the case where $X$ is a manifold, $r_i=0$ for all $i$. The set of invariants in this case coincides with the one defined by Raymond in \cite{R}. The definition of this set of invariants of $X^*$ suggests the following notion of equivalence between orbit spaces. 

\begin{definition}
Let  $S^1$ act effectively and isometrically on two closed, connected Alexandrov $3$-spaces $X$ and $Y$. We will say that their orbit spaces are \textit{isomorphic} if there is a \textit{weight-preserving homeomorphism} $X^{*} \rightarrow Y^{*}$, i.e., a homeomorphism preserving the orbit types and isotropy information of the orbit spaces. If $X^{*}$ and $Y^{*}$ are oriented, we also require the homeomorphism to be orientation-preserving.       
\end{definition}

We present the following result without a proof, as it is straightforward.

\begin{proposition}
Let $S^1$ act effectively and isometrically on two closed, connected Alexandrov $3$-spaces $X$ and $Y$. If $X$ and $Y$ are equivariantly homeomorphic, then $X^{*}$ and $Y^{*}$ are isomorphic. 
\end{proposition}

\section{Topological and equivariant classification when $X^{*}$ is a disk, $E=\emptyset$ and $s\geq 1$}
\label{S:TOP_EQU_CLASS_DISK}
 
We will first focus our attention on the case that $X^*$ is homeomorphic to a $2$-disk without exceptional orbits and at least two orbits of topologically singular points.  This is the simplest orbit space that can arise from a non-manifold Alexandrov space. We remark that these conditions imply that $SF\neq \emptyset$, and, in particular, that $F\neq \emptyset$. The main goal of this section will be to construct a cross-section to the orbit map when there are no exceptional orbits and use it to obtain a topological decomposition of $X$. We follow the same strategy as in the manifold case (see Section 3 in \cite{R} and Lemma 2 of Section 1.9 in \cite{O}). The existence of this cross-section will also yield a weakly equivariant classification of the effective, isometric $S^1$-actions on $X$, as is shown in Corollary \ref{COR:HOM_EQU}. When dealing with arbitrary permissible values for the invariants defined in the last section, the simpler case considered here will play a fundamental role.  Throughout this and the next section the term \textit{cross-section} will be used to refer to both a map $X^* \rightarrow X$ and its image on $X$.    

\begin{theorem}
\label{TH:CROSS_SECTION}
Let $S^{1}$ act effectively and isometrically on a closed, connected Alexandrov $3$-space $X$ that is not a manifold. Assume that there are no exceptional orbits and that $X^{*}$ is homeomorphic to a $2$-disk. Then there exists a cross-section to the orbit map. 
\end{theorem}  

\begin{proof}
Let $2r$ be the number of topologically singular points of $X$. We will proceed by induction on $r$.

We will first assume that $r=1$ and denote the topological singularities by $x^+$ and $x^-$. We will construct a cross-section $X^{*}\rightarrow X$ by decomposing $X$ into subsets admitting cross-sections. By Proposition \ref{PROP:ORBIT_SPACE} the boundary of $X^*$ is the union of two arcs $I_1\subset F^{*}$ and $I_2\subset SE^*$ such that $I_1\cap I_2 = \{(x^+)^*,(x^-)^* \}$. Let $\varepsilon>0$ be small enough so that $B_{\varepsilon}(x^+)$ and $B_{\varepsilon}(x^{-})$ are conical \cite{Pe}. By Theorem \ref{TH:SLICE_THEOREM} we may assume that a tubular neighborhood $U$ of $F\cup SE$ of radius $\varepsilon$ is invariant. Then, $U\setminus \left( B_{\varepsilon}(x^+) \cup B_{\varepsilon}(x^-) \right)$ is an invariant subset of $X$ consisting of two disjoint components. Let $U_{RF}$ and $U_{SE}$ be said components, so that, $U_{RF}^*$ and $U_{SE}^*$ intersect $I_1$ and $I_2$ respectively. Figure \ref{FIG:ORBIT_SPACE_DECOMPOSITION} depicts the induced decomposition on $X^{*}$. Let $\overline{U}$ be the closure of $U$. Observe that $P:= X\setminus \overline{U}$ is contained in the principal stratum of $X$. Furthermore, $P^{*}$ is contractible since it is homeomorphic to an open $2$-disk. Therefore, the restriction of the orbit map to $P$ is a trivial principal $S^1$-bundle. Thus, we have a cross-section $h_P: \overline{P^{*}} \rightarrow \overline{P}$. We will now show that this cross-section can be extended to $U^*$.

\begin{figure}
\centering
\def\svgwidth{0.5\columnwidth}
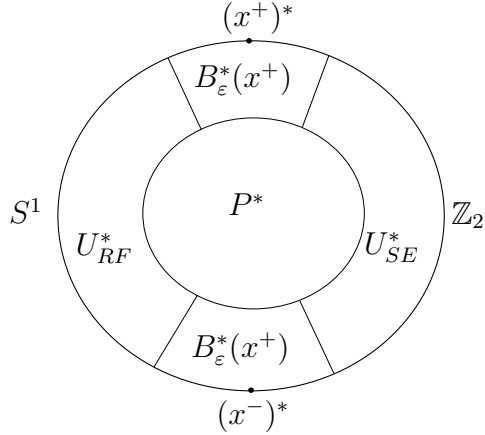
\caption{Decomposition of $X^*$ into neighborhoods with cross-sections.}
\label{FIG:ORBIT_SPACE_DECOMPOSITION}
\end{figure} 

We extend $h_P$ to $U_{RF}^*$ first. By Theorem \ref{TH:SLICE_THEOREM}, $U_{RF}$ is equivariantly homeomorphic to a solid tube $D^2\times I$ with an action by rotations around its axis $\{0\}\times I$. The common boundary of $P$ and $U_{RF}$ is a cylinder $C:=\sphere^1\times I$. We have a continuous curve $m$ on $C$ defined as $h_P(\overline{P^{*}}) \cap \partial U_{RF}$, where $\partial U_{RF}$ denotes the boundary of $U_{RF}$. Since $m$ is the restriction of $h_P$ to $C$, $(D^{2}\times \{t \})\cap m$ consists of exactly one point $m_t$ for each $t\in I$. Now, we connect $m_t$ with $(0,t)$ by a line segment. The resulting subset of $D^2\times I$ is a cross-section $h_{RF}: U_{RF}^{*} \rightarrow U_{RF}$. We observe that the restrictions of $h_{RF}$ and $h_P$ to $C$ coincide. 

 We extend $h_P$ to $U_{SE}^*$ similarly. By Theorem \ref{TH:SLICE_THEOREM} a small neighborhood of an orbit in $SE$ is equivariantly homeomorphic to $\sphere^1\times_{\bb{Z}_2} D^2$, the non-trivial $D^2$-bundle over $\sphere^1$. Consider $\Real P^2$ parametrized as in Example \ref{EX:STANDARD_ACTION} with the same circle action. Let $D_{\delta}^2\subset \Real P^2$ be the disk of radius $\delta<1$ centered at $[0,\theta]$. Then $\sphere^1\times_{\bb{Z}_2} D^2$ is equivariantly homeomorphic to $(\Real P^2\setminus D_{\delta}^2)\times I$ where the action on $I$ is trivial. Consequently, $U_{SE}$ is equivariantly homeomorphic to $(\Real P^2\times I) \setminus (D_{\delta}^2\times I)$. The common boundary between $U_{SE}$ and $P$ is again a cylinder $C$. As before, $h_P(\overline{P^{*}}) \cap \partial U_{SE}$ determines a continuous curve $l$ on $C$. Observe that each point $l_t$ of $l$ determines a unique point $([1,\theta_t], t )\in (\Real P^2\times I) \setminus (D_{\delta}^2\times I)$. Therefore, by joining $l_t$ with the corresponding point $([1,\theta_t], t )$, a cross-section $h_{SE}: U_{SE}^{*} \rightarrow U_{SE}$ is obtained. The restrictions of $h_{SE}$ and $h_P$ to $C$ coincide.    
  
 So far, we have a cross-section $h_0 : \overline{P^{*}} \cup U_{RF}^{*} \cup U_{SE}^{*} \longrightarrow \overline{P} \cup U_{RF} \cup U_{SE}$. We will extend $h_0$ to $B_{\varepsilon}(x^+)$. Recall that we assumed that $B_{\varepsilon}(x^+)$ is conical. Then by Theorem \ref{TH:SLICE_THEOREM}, $B_{\varepsilon}(x^+)$ is equivariantly homeomorphic to $K(\Real P^2)$ equipped with the standard circle action. Let $w$ be the curve given by $h_0(\overline{P^{*}} \cup U_{RF}^{*} \cup U_{SE}^{*}) \cap \partial B_{\varepsilon}(x^+)$. A cross-section to the action on  $B_{\varepsilon}(x^+)^*$ is obtained by repeating the curve $w$ on each level $\Real P^2 \times \{t\}$ of $B_{\varepsilon}(x^+)$. We extend $h_0$ to $B_{\varepsilon}(x^-)^*$ analogously. This concludes the proof of the theorem for $r=1$.
 
Suppose now that $r=k+1$. We assume that every effective, isometric circle action on a closed, connected $3$-space, with $2k$ topologically singular points, has a cross-section. Take two edges in $RF^{*}$ that are separated by a single edge in $SE^{*}$ and let $\gamma$ be a geodesic that connects them by arbitrary points. This separates $X^*$ into two subsets. Let $X^{*}_{2}$ be the subset of $X^{*}$ with two points in $SF^{*}$ and $X^{*}_{2k}$, the subset with $2k$ points in $SF^{*}$. Let $\pi:X\rightarrow X^*$ be the canonical projection. Then, $\pi^{-1}(\gamma)$ is an invariant $2$-sphere in $X$.  The invariant subspaces $X_2=\pi^{-1}(X^*_{2})$ and $X_{2k}=\pi^{-1}(X^{*}_{2k})$ of $X$, share $\pi^{-1}(\gamma)$ as boundary. Observe that the restriction of the action to $\pi^{-1}(\gamma)$ is equivalent to an orthogonal action \cite{M}. Let $B$ be a closed $3$-ball with the orthogonal $S^1$-action and let $B^{*}$ be its orbit space. The weights on $B^{*}$ are as follows. The interior of $B^{*}$ corresponds to principal isotropy. Its boundary is composed of two arcs, one of principal isotropy and the other one of fixed points. Denote the boundary arc of principal isotropy by $\tilde{\gamma}$. Let $F:\pi^{-1}(\gamma) \rightarrow \partial B$ be an equivariant homeomorphism and $f:\gamma \rightarrow \tilde{\gamma}$ a homeomorphism. The spaces $\tilde{X_2}:= X_2\cup_{F}B$ and $\tilde{X_{2k}}:=X_{2k}\cup_{F}B$ are naturally endowed with effective, isometric $S^1$-actions. Furthermore, their orbit spaces are isomorphic to the  topological surfaces $\tilde{X}^{*}_2 := X^{*}_2 \cup_{f}B^*$ and $\tilde{X}^{*}_{2k}:=X^{*}_{2k}\cup_{f}B^*$ respectively. We note that $\tilde{X_2}$ and $\tilde{X_{2k}}$ have $2$ and $2k$ topologically singular points respectively. By our induction hypothesis and the case $r=1$, there exist cross-sections $\tilde{h}_{2}: \tilde{X}^{*}_2 \rightarrow \tilde{X}_2$ and $\tilde{h}_{2k}: \tilde{X}^{*}_{2k} \rightarrow \tilde{X}_{2k}$.  We restrict $\tilde{h}_2$ and $\tilde{h}_{2k}$ to obtain cross-sections $h_2: X^{*}_2 \rightarrow X_2$ and $h_{2k}: X^{*}_{2k} \rightarrow X_{2k}$. We make $h_{2}$ and $h_{2k}$ coincide on $\pi^{-1}(\gamma)$ by means of an equivariant homeomorphism $\pi^{-1}(\gamma) \rightarrow \pi^{-1}(\gamma)$ to obtain a global cross-section $h : X^{*} \rightarrow X$.  
\end{proof}

We obtain the following Corollary to Theorem \ref{TH:CROSS_SECTION}.

\begin{corollary}
\label{COR:HOM_EQU}
Let $S^1$ act effectively and isometrically on two closed, connected Alexandrov $3$-spaces $X$ and $Y$ that are not manifolds. Assume that the actions have no exceptional orbits and that the orbit spaces $X^{*}$ and $Y^{*}$ are homeomorphic to $2$-disks. Then $X$ is weakly equivariantly homeomorphic to $Y$ if and only if $X^{*}$ is isomorphic to $Y^{*}$.   
\end{corollary}

\begin{proof}
Let $\pi_X:X\rightarrow X^*$ and $\pi_Y:Y\rightarrow Y^*$ be the canonical projections. By Theorem \ref{TH:CROSS_SECTION}, there exist cross-sections $h_X: X^{*} \rightarrow X$ and $h_Y: Y^{*} \rightarrow Y$. We let $\Psi: X^{*} \rightarrow Y^{*}$ be an isomorphism and define $\tilde{\Psi} = h_Y\circ\Psi\circ\pi_X$. The function $\tilde{\Psi}$ takes $h_X(X^{*})$ onto $h_Y(Y^{*})$ homeomorphically. The equivariance of $\tilde{\Psi}$ follows from the injectivity of $\tilde{\Psi}^{-1}$, noting that $\tilde{\Psi}^{-1}(\tilde{\Psi}(gx))=\tilde{\Psi}^{-1}(f(g)\tilde{\Psi}(x))$ for every $g\in S^1$, $x\in X$ and every automorphism $f$ of $S^1$. 

We construct a weakly equivariant homeomorphism $\Phi : X \rightarrow Y$ in the following manner. For each $x\in X$ there is a unique representation of the form $gh_X(x_0^*)$. Thus, $\Phi(gh_X(x_0^*)) := f(g) \tilde{\Psi}(h_X(x_0^*))$ is a weakly equivariant homeomorphism. Its inverse is obtained similarly by noting that $\Psi^{-1}(gh_Y(y_0^{*}))= f^{-1}(g)\tilde{\Psi}^{-1}(h_Y(y_0^*))$.
\end{proof}

\subsection{Equivariant connected sums}

 Let $S^1$ act effectively and isometrically on two closed, connected Alexandrov $3$-spaces $X_1$ and $X_2$. Let $RF_i$ denote the set of topologically regular fixed points of $X_i$, $i=1,2$, and assume that $RF_i\neq\emptyset$. We will define an \textit{equivariant connected sum} $X_1\#X_2$. This construction extends the equivariant connected sum for $3$-manifolds with circle actions defined by Raymond in Section 8 of \cite{R}. 
 
Let $x_i\in RF_i\subset X_i$ be a topologically regular fixed point.  Let $B_i$ be a small invariant open neighborhood of $x_i$ and let $\overline{B}_i$ be the closure of $B_i$. We may assume that $B_i$ is small enough so that $\overline{B}_i\cap SF_i=\overline{B}_i\cap E_i=\overline{B}_i\cap SE_i=\emptyset$, i.e., $\overline{B}_i$ contains only topologically regular fixed points and points with trivial isotropy. Since the point $x_i$ is topologically regular we may assume that $\overline{B}_i$ is equivariantly homeomorphic to a closed $3$-ball with an effective and isometric circle action. By Theorem \ref{TH:SLICE_THEOREM}, the $S^1$-action on $\overline{B}_i$ must be equivariantly homeomorphic to the cone over an $S^1$-action on $\partial \overline{B}_i \cong \mathbb{S}^2$. Since every topological action of $S^1$ on $\mathbb{S}^2$ is equivariantly homeomorphic to a linear action (see \cite{M}, \cite{N}), it follows that the action of $S^1$ on $\mathbb{S}^2$ is equivariantly homeomorphic to the action given by rotations of a round $2$-sphere around an axis. Hence the $S^1$-action on $\overline{B}_i$ is linear and fixes an arc joining two points in $\partial \overline{B}_i$ (corresponding to the fixed points of the $S^1$-action on $\partial \overline{B}_i$) and whose interior is contained in $B_i$. In particular, this arc contains the point $x_i$ (see Raymond \cite{R} and Orlik and Raymond \cite{OR} for a purely topological proof of the fact that an $S^1$-action on $\overline{B}_i$ must be linear).
 
We now let $X'_i:=X_i\setminus B_i$ and glue $X'_1$ to $X'_2$ by means of an equivariant homeomorphism. The equivariant homeomorphism is required to be orientation reversing if the $X_i$ are orientable. We obtain a topological space $X$ carrying an effective $S^1$-action by homeomorphisms.

 \begin{definition} The space $X$ constructed from $X_1$ and $X_2$ in the preceding paragraph is the \textit{equivariant connected sum} of $X_1$ and $X_2$ and it is denoted by $X_1\# X_2$.
 \end{definition}

This construction can be iterated to obtain an equivariant connected sum of any finite number of connected summands. Let $X_1,\ldots, X_n$ be closed, connected Alexandrov $3$-spaces on which $S^1$ acts effectively, isometrically with topologically regular fixed points. We now indicate how to equip an equivariant connected sum $X_1\#\cdots \#X_n$ with an Alexandrov metric such that the $S^1$-action on $X$, induced by each $X_i$, is isometric.

\begin{proposition}Let $S^1$ act effectively and isometrically on $n\geq 2$ closed, connected Alexandrov $3$-spaces $X_i$, $i=1,\ldots, n$, and let $X$ be the equivariant connected sum $X_1\#\cdots \#X_n$. Then there exists an Alexandrov metric on $X$. Furthermore, this metric is invariant with respect to the effective and isometric circle action on $X$ induced by the circle actions on each $X_i$.  
\end{proposition}

\begin{proof}
  If $X$ is a topological manifold, then by Theorem 6 of \cite{R} there is a differentiable structure on $X$ such that the $S^1$-action is equivalent to an action by diffeomorphisms. Since $X$ is compact, the circle action on $X$ induced by the $X_i$ is proper. Therefore, there is a Riemannian metric $g$ on $X$ such that the elements of $S^1$ are isometries with respect to $g$. The metric on $X$ induced by $g$ is an Alexandrov metric since $X$ is compact. 
 
 Now, assume that $X$ is not a topological manifold. We consider the \textit{ramified orientable double cover} $\tilde{X}$ of $X$. In order to keep the presentation self-contained, we will describe $\tilde{X}$ by means of elementary tools. However, we refer the reader to \cite[Theorem 2.4]{HS} for a more general construction. We proceed with the description of $\tilde{X}$. We remove from $X$ disjoint, open conical neighborhoods of each topologically singular point, obtaining a non-orientable topological $3$-manifold $X_0$ with boundary an even number of copies of $\Real P^2$. The orientable double cover $\tilde{X}_0$ of $X_0$ is an orientable, topological $3$-manifold with boundary. By Theorem 9.1 in \cite{Br} we can lift the $S^1$-action on $X_0$ to obtain an effective $S^1$-action by homeomorphisms on $\tilde{X}_0$. We also note that the $S^1$ acting on $\tilde{X}_0$ is a $2$-fold covering of the $S^1$ acting on $X_0$. We let $\xi: S^1 \rightarrow S^1$ be said covering. We need some technical facts. Let $\iota$ be the natural involution on $\tilde{X}_0$ and $\rho: \tilde{X}_0 \rightarrow \tilde{X}_0/\iota$, the canonical projection. First we observe that, since $\rho$ is 2-sheeted, then $\mathrm{Aut}(\rho)$, the group of deck transformations of $\rho$, is isomorphic to $\mathbb{Z}_2$. We also observe that $\iota$ is an element of $\mathrm{Aut}(\rho)$. By Theorem 9.1 in \cite{Br}, the kernel of $\xi$ is a subgroup of $\mathrm{Aut}(\rho)$. Therefore $\iota$ coincides with the function $\{e^{\pi}\}\times \tilde{X}_0 \rightarrow \tilde{X_0}$, the restriction of the $S^1$ action. Since each boundary component of $\tilde{X}_0$ is a $2$-sphere, the restriction of the $S^1$-action is orthogonal \cite{M}. Then we can extend $\iota$ and the $S^1$ action to $3$-balls to obtain a closed topological $3$-manifold $\tilde{X}$. Note that $\tilde{X}/\iota$ is homeomorphic to $X$. Now, we apply Theorem 6 of \cite{R} to conclude that the circle action on $\tilde{X}$ is equivalent to an action by diffeomorphisms. This also implies that the action of $\iota$ on $\tilde{X}$ is equivalent to an action by diffeomorphisms. Furthermore, the smoothed actions of $\iota$ and $S^1$ commute. Now we let $\tilde{g}$ be a Riemannian metric on $\tilde{X}$ such that the $S^1$ and $\iota$ actions are isometric. Then, $(\tilde{X},\tilde{g})/\iota$ is a Riemannian orbifold with an effective, isometric $S^1$-action equivalent to that induced by the $X_i$. 
\end{proof}

In particular, we have the following observation. Let $\sphere^2$ be the unit round $2$-sphere and consider $\sphere^3 = \Susp(\sphere^2)$ with the standard spherical suspension metric. Let $\iota:\sphere^3 \rightarrow \sphere^3$ be given by the antipodal map on each level of the suspension. Then $(\Susp(\Real P^2), d_0)$ is isometric to the quotient of the unit round $\sphere^3$ by $\iota$. Therefore, $(\Susp(\Real P^2),d_0)$ has the structure of a Riemannian orbifold with curvature bounded below and has an effective, isometric $S^1$-action. Thus, the connected sum of finitely many copies of $\Susp(\Real P^2)$ has an Alexandrov metric and the $S^1$-action determined by taking the standard action on every summand is effective and isometric. Thus, we obtain the following corollary.

\begin{corollary}
Let $S^1$ act effectively and isometrically on a closed, connected Alexandrov $3$-space $X$ with $2r$ topologically singular points, $r\geq 1$. If there are no exceptional orbits and $X^{*}$ is homeomorphic to a $2$-disk, then $X$ is weakly equivariantly homeomorphic to the equivariant connected sum of $r$ copies of $\Susp(\Real P^2)$ equipped with the standard circle action. Consequently, the only effective, isometric circle action on $\Susp(\Real P^2)$ is the standard action, up to weakly equivariant homeomorphism.
\end{corollary}

\begin{remark}
We can avoid the use of the Slice Theorem for Alexandrov spaces (Theorem \ref{TH:SLICE_THEOREM}) in our present setting as follows. Observe that an invariant conical neighborhood of $x\in SF$ is homeomorphic to $K(\Real P^2)$ \cite{Pe}. There exists a topological  involution $\iota$ of the $3$-ball $B$, such that $B/\iota$ is homeomorphic to $K(\Real P^2)$. By results of Hirsch and Smale \cite{HiS} and Livesay \cite{L}, the action of the involution must be orthogonal. Hence this action is the cone of the action induced by the antipodal map on $\mathbb{S}^2$.  On the other hand, the action of $S^1$ on $B$ is equivalent to an orthogonal action \cite{M}. Since these actions on $B$ commute, we have that the action of $S^1$ on $K(\Real P^2)$ is the cone of the standard action on $\mathbb{R}P^2$. For a more general instance of this construction in Alexandrov geometry, see for example, Section 2 of \cite{GW}, Section 2 of \cite{HS} or Lemmas 1.6 and 1.7 of \cite{GG1}.  
\end{remark}

\section{Topological and equivariant classification in the general case}
\label{S:CLASS_GENERAL}

In this section we will prove Theorem \ref{teoprinc}. To this end we will consider effective, isometric circle actions on $X$ with $F\neq \emptyset$ having no restrictions on the orbit space.  The proof will follow along the lines of the proof in the manifold case (see \cite{O,OR, R}). It consists of first obtaining a cross-section to the action everywhere except for a tubular neighborhood of $E$ and then noting that one can define a global weakly equivariant homeomorphism between spaces with isomorphic orbit spaces. This cross-section will be constructed by using the more restrictive case considered in the previous section.  Then, one must use the fact that, just as in the manifold case,  there is essentially a unique way to glue a tubular neighborhood of an exceptional orbit once the restriction of a cross-section to the boundary and the Seifert invariants of the orbit are given.   

\begin{proposition}
\label{PROP:CROSS_SEC_GRL}
Let $S^1$ act effectively and isometrically on a closed, connected Alexandrov $3$-space $X$ with $E=\emptyset$ and $F\neq\emptyset$. Then there exists a cross-section to the action.  
\end{proposition} 

\begin{proof} 
Let $( b; (\varepsilon, g, f,t); (r_1,r_2, \ldots, r_s))$ be the invariants of the action. If $X$ has no topologically singular points then $X$ is a topological manifold, and the result is Lemma 2 in \cite{R}.
Thus, we first assume that $s=1$ and denote $r_1=r$. Consider the topological surface $M^{*}$ weighted by the tuple $\left(  b; (\varepsilon, g, f+1,t) \right)$. By Theorem 4 in \cite{R}, there is an effective, isometric $S^1$-action on a closed $3$-manifold $M$ with $M^*$ as the orbit space. Furthermore, $M$ is unique up to weakly equivariant homeomorphism. Since $f+1>0$, there is at least one circle $C$ of fixed points on $M$. Consider an arc $I$ contained in $C$. Let $U$ be a small tubular neighborhood of $I$. Now, let $\tilde{X}$ be the equivariant connected sum of $r/2$ copies of $\Susp(\Real P^2)$. Take an edge of fixed points in $\tilde{X}$ that has topologically singular points as endpoints and let $\tilde{I}$ be a subarc of said edge consisting of topologically regular points only. Consider a small tubular neighborhood $\tilde{U}$ of $\tilde{I}$. By Theorem \ref{TH:SLICE_THEOREM} the restricted actions on $U$ and $\tilde{U}$ are equivalent to an action by rotations with respect to $I$ and $\tilde{I}$ respectively. Thus, there is an equivariant homeomorphism $\varphi : \tilde{U} \rightarrow U$. We now take the equivariant connected sum $M\# \tilde{X}= M\cup_{\varphi} \tilde{X}$. We then have that $(M\# \tilde{X})^{*}$ is isomorphic to $M^{*} \cup \tilde{X}^{*}$, gluing along $\tilde{U}^*$ and $U^{*}$. Observe that $(M\# \tilde{X})^{*}$ is also isomorphic to $X^{*}$. The subsets $\pi^{-1}(M^{*})$ and $\pi^{-1}(\tilde{X}^{*})$ are invariant in $X$. Moreover, $\pi^{-1}(M^{*}) \cap SF = \emptyset$, and therefore, $\pi^{-1}(M^{*})$ is a topological $3$-manifold. We conclude that $M$ is weakly equivariantly homeomorphic to $\pi^{-1}(M^{*})$.
  
By Lemma 2 in \cite{R} and Theorem \ref{TH:CROSS_SECTION}, we have cross-sections $h_1:M^{*} \rightarrow M$ and $h_2:\tilde{X}^{*} \rightarrow \tilde{X}$. As mentioned in the preceding paragraph, the restricted actions on $\tilde{U}$ and $U$ are equivalent to an orthogonal action on a $3$-ball $B$. This action has a canonical cross-section $J\subset B^{3}$. We take equivariant homeomorphisms $\varphi_1: U\rightarrow B$ and $\varphi_2: B\rightarrow \tilde{U}$ such that $\varphi_1$ and $\varphi_2$ take $h_1(M^{*})$ and $J$ homeomorphically onto $J$ and $h_2(\tilde{X}^{*})$, respectively. Therefore, the equivariant homeomorphism $\varphi_2\circ\varphi_1$ makes $h_1$ and $h_2$ agree. Then, we obtain a global cross-section $h:X^*\rightarrow X$. This concludes the proof of the Proposition for $s=1$.    

For the general case, we let $M^*$ be weighted by $\left(b; \left( \varepsilon, g, f+s, t \right) \right)$. We use Theorem 4 in \cite{R} again to obtain the unique closed $3$-manifold $M$. In this case, $M$ has at least $s$ circles of fixed points. We let $\tilde{X}_i$ be the equivariant connected sum of $r_i/2$ copies of $\Susp(\Real P^2)$, for each $i=1,2,\ldots, s$. Then $X^{*}$ is isomorphic to $M^*\cup\left( \cup_{i=1}^s \tilde{X}_i^* \right)$ , where the unions are taken along adequate invariant neighborhoods of the fixed point components. Applying the procedure made in the case $s=1$ for each circle of fixed points, we get cross-sections $M^{*}\rightarrow M$ and $\tilde{X_i}^{*} \rightarrow \tilde{X}_i$. We glue these cross-sections to obtain a global cross-section $X^{*}\rightarrow X$.
\end{proof}

\begin{proof}[Proof of Theorem \ref{teoprinc}] 
If $X$ has no topologically singular points, then the result reduces to Corollaries 2a, 2b and Theorems 1 and 4 in \cite{R}. Therefore we now assume that $X$ has topologically singular points.

We will prove (2) first. Let $X_0$ denote the complement in $X$ of a sufficiently small tubular neighborhood of $E$, so that $X_0^{*}$ is homeomorphic to $X^{*}$ with $n$ disks removed. By Proposition \ref{PROP:CROSS_SEC_GRL} there is a cross-section $X_0^{*} \rightarrow X_0$. Let $Y$ be a closed, connected Alexandrov $3$-space with an effective, isometric $S^1$-action such that $X_0^*$ and $Y_0^*$ are isomorphic. By replicating the argument in Corollary \ref{COR:HOM_EQU}, we obtain a weakly equivariant homeomorphism $X_0 \rightarrow Y_0$. In the notation of Proposition \ref{PROP:CROSS_SEC_GRL}, $X_0^*$ is isomorphic to the orbit space $M_0^*\cup\left( \cup_{i=1}^s \tilde{X}_i^* \right)$, where $M_0$ has no exceptional orbits and has $n$ torus boundary components. Therefore, there exists a weakly equivariant homeomorphism $\varphi: X_0 \rightarrow M_0\#\Susp(\bb{R}P^2)\#\ldots \#\Susp(\bb{R}P^2)$, where the connected sum has $s$ summands equal to $\Susp(\bb{R}P^2)$. We now observe that Lemma 6 and Theorems 2a and 2b in \cite{R} admit straightforward generalizations to the Alexandrov setting by using our Theorem \ref{TH:CROSS_SECTION}. Hence, as in the manifold case, $\varphi$ can be extended to a weakly equivariant homeomorphism between $X$ and $Y$. 

We now prove (1). The restriction of the action to the manifold $M_0$ appearing on the previous decomposition of $X$ is uniquely determined, up to weakly equivariant homeomorphism, by Theorem 4 in \cite{R}. On the other hand, the restriction of the action to $\Susp(\bb{R}P^2)\#\ldots \#\Susp(\bb{R}P^2)$ is an equivariant connected sum of standard actions.  Therefore, the action is determined by the number of pairs of topologically singular points on each boundary component of $X^*$.
\end{proof}

\begin{remark}
Recall that $s$ is the number of boundary components of $X^*$ which intersect $SF^*$. The set of invariants $\left(b; (\varepsilon, g, f, t); \{ \alpha_i, \beta_i \}_{i=1}^{n}; s \right)$ provides enough information to obtain the topological decomposition of $X$. However, by excluding the $s$-tuple $(r_1,r_2, \ldots, r_s)$, the remaining invariants are incapable of detecting some inequivalent actions on $X$ if the number of topologically singular points is greater than $2$, as the following example shows. 
\end{remark}

\begin{example}
\label{EX:INEQU_ACTIONS}
Let $M= \sphere^2 \times \sphere^1$, regarding $\sphere^2$ as a subset of $\mathbb{C}\times \Real$. Consider the $S^1$-action on $M$ that sends each $(z,t,w)\in \sphere^2 \times \sphere^1$ to $(gz,t,w)$, where $g\in S^1$ and $gz$ is the complex multiplication. Let $X_1$ and $X_2$ denote two copies of $\Susp(\Real P^2)$ equipped with the standard circle action. The equivariant connected sum $X= M \# X_1 \# X_2$ is realized by choosing small tubular neighborhoods of subarcs of the components of fixed points of the connected summands. Observe that $M$ has two circles of fixed points, namely, $C_1=\{(0,1)\}\times \sphere^1$ and $C_2=\{(0,-1)\}\times \sphere^1$. Note that each $X_i$ has one fixed point component, which we will denote by $F_1$ and $F_2$, respectively. Therefore the choices involved in the construction of the equivariant connected sum can be done in two ways. On the one hand, we can glue $F_i$ to subarcs of $C_1$, obtaining an orbit space $X^*$ with $C_2^*\cap SF^*=\emptyset$. On the other hand, we can glue $F_1$ with a subarc in $C_1$, and $F_2$ with a subarc in $C_2$. In the resulting orbit space, $C_i\cap SF^*\neq \emptyset$. These actions on $X$ cannot be equivalent since their orbit spaces are not isomorphic.  
\end{example}

\begin{remark} 
\label{R:COUNT_ACTIONS}
Example \ref{EX:INEQU_ACTIONS} illustrates how we count the number of inequivalent effective, isometric circle actions on $X$. By Theorem \ref{teoprinc}, $X$ is weakly equivariantly homeomorphic to $M\#Y$, where $Y$ is an equivariant connected sum of $s$ copies of $\Susp(\Real P^2)$ and $M$ is a closed $3$-manifold. Since $Y$ can only contribute standard circle actions, we only have to choose how to arrange $s$ pairs of topologically singular points along the boundary components of mixed isotropy in $X^{*}$. Following the notation of Theorem \ref{teoprinc}, if $s>0$, then there are $\genfrac{(}{)}{0pt}{}{r}{s}$ inequivalent effective, isometric circle actions for each effective, isometric circle action on $M$.
\end{remark}

\section{Borel conjecture for Alexandrov spaces with circle symmetry}
\label{S:Borel}

The simplest examples of Alexandrov spaces which are not manifolds occur within the class of closed Alexandrov $3$-spaces, since they are topological manifolds except for a finite number of isolated points. This property suggests that some results for closed $3$-manifolds may have suitable generalizations to Alexandrov $3$-spaces. 

Recall that a topological space $X$ is said to be \textit{aspherical} if its homotopy groups $\pi_{q}(X)$ are trivial for $q>1$. One result concerning the class of aspherical $n$-manifolds is the \emph{Borel conjecture}. It asserts that if two closed, aspherical $n$-manifolds, are homotopy equivalent, then they are homeomorphic. The proof of this conjecture in the $3$-dimensional case is a consequence of Perelman's proof of Thurston's Geometrization Conjecture (see~\cite{P}). It is natural to ask if this conjecture still holds for closed, connected Alexandrov $3$-spaces, particularly for those with symmetry. The explicit topological decomposition in Theorem \ref{teoprinc} allows us to investigate the homotopy groups of these spaces and to prove the following analog of the Borel conjecture.    

\begin{theorem}[Borel conjecture for Alexandrov spaces with circle symmetry] If two aspherical, closed, connected Alexandrov $3$-spaces on which $S^1$ acts effectively and isometrically are homotopy equivalent, then they are homeomorphic.  
\end{theorem}

\begin{proof}
Our proof will consist of showing that the only aspherical, closed, connected Alexandrov $3$-spaces admitting an effective, isometric $S^1$-action are topological manifolds. As pointed out before, the Borel Conjecture holds for closed, aspherical $3$-manifolds \cite{P}.  

We begin by noting that $\Susp(\Real P^2)$ is not aspherical: a combination of the suspension isomorphism and the Hurewicz Theorem yields that $\pi_2(\Susp(\Real P^2))\cong \mathbb{Z}_2$. We will now prove that a connected sum of suspensions of $\Real P^2$ is not aspherical. We use homology with $\bb{Z}$ coefficients. Let $X = \Susp(\Real P^2)\# \Susp(\Real P^2)$ and $B\subset \Susp(\Real P^2)$ be an invariant $3$-ball used for the construction of the equivariant connected sum. By the Seifert-Van Kampen Theorem, $X$ is simply-connected. Therefore, by the Hurewicz Theorem, $\pi_2(X)\cong H_2(X)$. Observe that $\partial B\cong\sphere^2$ is a deformation retract of a neighborhood in $X$. Hence, by Proposition 19.36 of \cite{GH},  $H_2(X,\sphere^2)\cong H_2(\Susp(\Real P^2)\vee \Susp(\Real P^2))$. Also note that the distinguished point in $\Susp(\Real P^2)\vee \Susp(\Real P^2)$ is a deformation retract of neighborhoods $U_1$ in the first $\Susp(\Real P^2)$ and $U_2$ in the second $\Susp(\Real P^2)$. Then, by applying the Mayer-Vietoris sequence to the decomposition $(\Susp(\Real P^2)\cup U_2, U_1\cup \Susp(\Real P^2) )$, we obtain that $H_2(\Susp(\Real P^2)\vee \Susp(\Real P^2)) \cong \mathbb{Z}_2 \oplus \mathbb{Z}_2$. Hence, the exact sequence of the pair $(X,\sphere^2)$ takes the following form
\[
 H_2(X) \rightarrow H_2(\Susp(\Real P^2)\vee \Susp(\Real P^2)) \rightarrow H_1(\sphere^2).  
\]
  Therefore, we have a surjection $H_2(X) \rightarrow \mathbb{Z}_2 \oplus \mathbb{Z}_2$. It follows that $H_2(X)\neq 0$ and, by induction, that no connected sum of finitely many  suspensions of $\Real P^2$ is aspherical.

Let $Y$ be an aspherical, closed, connected Alexandrov $3$-space on which $S^1$ acts effectively and isometrically. By Theorem \ref{teoprinc}, $Y$ is homeomorphic to $M\#X$, where $X$ is a connected sum of finitely many copies of $\Susp(\Real P^2)$ and $M$ is a closed $3$-manifold. Let $\varphi: M\#X \rightarrow M\vee X$ be the function that collapses the $\sphere^2$ used to construct the connected sum to a point. The pair $(M\# X, \sphere^2)$ is $0$-connected, therefore by Proposition 4.28 in \cite{H}, $\varphi$ is $2$-connected. Now, we lift $\varphi$ to the universal covers to get a $2$-connected map $\tilde{\varphi}: \widetilde{M\#X} \rightarrow \widetilde{M\vee X}$. Since we assumed $Y$ to be aspherical, $\pi_k(M\# X)=0$ for $k> 1$. Therefore, $\pi_k(\widetilde{M\#X})=0$ for $k\geq 1$. The map $\tilde{\varphi}$ and the Hurewicz Theorem yield that $\pi_2(\widetilde{M\vee X})=H_2(\widetilde{M\vee X})=0$.\\
Denote the projection of the universal cover of $M\vee X$ by $p$. Observe that $p^{-1}(M)\cap p^{-1}(X)=p^{-1}(\{pt\})$ is a discrete set and that $p^{-1}(M)\cup p^{-1}(X) = \widetilde{M\vee X}$. Using the Mayer-Vietoris sequence for this decomposition we obtain that $H_2(p^{-1}(M))\oplus H_2(p^{-1}(X))\cong 0$. The preimage of $X$ is a disjoint union of copies of the universal cover of $X$. Since $X$ is simply-connected, $p^{-1}(X)$ is a disjoint union of copies of $X$. This is a contradiction, since we proved that $H_2(X)\neq 0$.      
\end{proof}



\begin{thebibliography}{99}

\bibitem{Be} 
\textsc{V.~N.~Berestovski\v{\i}},
Homogeneous manifolds with an intrinsic metric II, Siberian Math. J. 30, no. 2, (1989), 180--191.

\bibitem{Br} 
\textsc{G.~E.~Bredon},  
Pure and  Applied Mathematics, Introduction to compact transformation groups, Academic Press, New York-London, 1972.

\bibitem{BBI}
\textsc{D.~Burago,  Y.~Burago and S.~Ivanov},
A course in metric geometry, Graduate Studies in Mathematics, American Mathematical Society, Providence, 2001.

\bibitem{BGP} 
\textsc{Y.~Burago, M.~Gromov and G.~Perelman},
A.D.~Aleksandrov's spaces with curvatures
bounded from below, Russian Math. Surveys 47, no. 2, (1992), 1--58.   

\bibitem{DW}
\textsc{D.~van Dantzig and B.~L.~van der Waerden},
\"Uber metrisch homogene R\"aume, (German) Abh. Math. Sem. Univ. Hamburg 6 (1928), 374--376.

\bibitem{FY}
\textsc{K.~Fukaya and T.~Yamaguchi},
Isometry groups of singular spaces, Math. Z. 216 (1994), 31--44.

\bibitem{GG1}
\textsc{F.~Galaz-Garcia and L.~Guijarro},
On three-dimensional Alexandrov spaces, Int. Math. Res. Not. IMRN, no. 14, (2015), 5560--5576.

\bibitem{GG} 
\textsc{F.~Galaz-Garcia and L.~Guijarro},
Isometry groups of Alexandrov spaces, Bull. London Math. Soc. 45,  no. 3, (2013), 567--579.

\bibitem{GS}
\textsc{F.~Galaz-Garcia and C.~Searle},
Cohomogeneity one Alexandrov spaces, Transform. Groups 16, no. 1, (2011), 91--107.

\bibitem{GH}
\textsc{M.~Greenberg and J.~Harper},
Mathematics Lecture Notes Series, Algebraic Topology A First Course, Benjamin-Cummings Publishing Company, Massachusetts, 1981.  

\bibitem{G} 
\textsc{K.~Grove},
Geometry of and via symmetries, Conformal, Riemannian and Lagrangian geometry (Knoxville, TN, 2000), Univ. Lecture Ser. 27, Amer. Math. Soc., Providence, (2002), 31--53.

\bibitem{GW}
\textsc{K.~Grove and B.~Wilking},
A knot characterization and $1$-connected nonnegatively curved $4$-manifolds with circle symmetry, Geom. Topol. 18, no. 5 (2014), 3091--3110.


\bibitem{HS} 
\textsc{J.~Harvey and C.~Searle},
Orientation and Symmetries of Alexandrov Spaces with Applications in Positive Curvature, arXiv:1209.1366v2 [math.DG], (2012).

\bibitem{H}
\textsc{A.~Hatcher},
Algebraic Topology, Cambridge University Press, Cambridge, (2002).

\bibitem{HiS}
\textsc{M.~Hirsch and S.~Smale},
On involutions of the $3$-sphere, Amer. J. Math. 81, (1959), 893-900.

\bibitem{K}
\textsc{V.~Kapovitch},
Perelman's stability theorem, Surv. Differ. Geom. 11, Int. Press, Somerville, (2007), 103--136. 

\bibitem{Ko}
\textsc{S.~Kobayashi}, 
Classics in Mathematics, Transformation groups in differential geometry, Springer-Verlag, Berlin, 1995.

\bibitem{L} 
\textsc{G.~R.~Livesay},
Involutions with two fixed points on the three-sphere, Ann. of Math. (2) 78, (1963) 582--593.

\bibitem{MY}
\textsc{A.~Mitsuishi and T.~Yamaguchi},
Collapsing three-dimensional closed Alexandrov spaces with a lower curvature bound, Trans. Amer. Math Soc., 367, no. 4, (2015), 2339--2410. 

\bibitem{M}
\textsc{P.~S.~Mostert},
On a compact Lie group acting on a manifold, Ann. Math. 65, (1957), 447--455. 

\bibitem{N}
\textsc{W.~D.~Neumann},
$3$-dimensional $G$-manifolds with $2$-dimensional orbits, Proc. Conf. on Transformation Groups (Ed. P.~S.~Mostert), (1968), 220--222. 

\bibitem{O}
\textsc{P.~Orlik},
Lecture Notes in Math., Seifert Manifolds, Springer-Verlag, Heidelberg, 1972. 

\bibitem{OR}
\textsc{P.~Orlik and F.~Raymond},
Actions of $SO(2)$ on $3$-manifolds, Proc. Conf. on Transformation Groups (Ed. P.~S.~Mostert), (1968), 297--318.

\bibitem{Pe}
\textsc{G.~Perelman},
Alexandrov's spaces with curvatures bounded from below II, preprint (1991), available at http://www.math.psu.edu/petrunin/papers/

\bibitem{P}
\textsc{J.~Porti},
Geometrization of three manifolds and Perelman's proof, Rev. R. Acad. Cienc. Exactas F\'is. Nat. Ser. A Math. RACSAM 102, no.1, (2008), 101--125.

\bibitem{R}
\textsc{F.~Raymond},
Classification of the actions of the circle on $3$-manifolds, Trans. Amer. Math. Soc. 131, (1968), 51--78.

\bibitem{S}
\textsc{C.~Searle}, 
Lecture Notes in Math., Geometry of Manifolds with Non-negative Sectional Curvature (Ed.\ R.~Herrera and Ed. L.~H.~Lamoneda), An introduction to isometric group actions with applications to spaces with curvature bounded from below, Springer, 2014, 21--43.


\end{thebibliography}
\end{document}